\documentclass{amsart}
\usepackage{amscd}
\usepackage{amssymb}
\usepackage{hyperref}
\usepackage[top=1.2in, bottom=1.2in, left=1.2in, right=1.2in]{geometry}

\input xy
\xyoption{all}

\def\Pprime{\ocoker'}

\def\naive{{\text{naive}}}
\def\exact{{\text{exact}}}

\def\cube#1#2#3#4#5#6#7#8{
& #5 \ar[rr] \ar[dl] \ar@{-}[d] && #6 \ar[dd] \ar[dl] \\
#1 \ar[rr] \ar[dd]  & \ar[d] & #2 \ar[dd] \\
& #7 \ar@{-}[r] \ar[dl] & \ar[r] & #8 \ar[dl] \\
#3 \ar[rr] && #4 \\
}

\def\HomMF{\operatorname{\uHom_{MF}}}

\def\rDsg{\mathsf{D}^{\mathsf{rel}}_{\mathsf{sg}}}
\def\rpDsg{{}^{\mathsf{p}}\mathsf{D}^{\mathsf{rel}}_{\mathsf{sg}}}
\def\Dsg{\mathsf{D}_{\mathsf{sg}}}
\def\Dsing{\Dsg}

\def\cU{\mathcal U}

\def\Tot{\operatorname{Tot}}

\def\H{{\mathbb H}}

\def\ker{\operatorname{ker}}
\def\cone{\operatorname{cone}}

\def\cP{\mathcal P}
\def\cE{\mathcal E}
\def\cL{\mathcal L}

\def\cB{\mathcal B}
\def\cC{\mathcal C}
\def\cZ{\mathcal Z}

\def\cF{\mathcal F}
\def\cG{\mathcal G}
\def\cO{\mathcal O}

\def\cT{\mathcal T}
\def\cM{\mathcal M}
\def\cN{\mathcal N}

\def\coker{\operatorname{coker}}
\def\ocoker{\operatorname{\overline{coker}}}

\def\dm{\operatorname{dim}}

\def\Proj{\operatorname{Proj}}
\def\uHom{\operatorname{\underline{Hom}}}

\def\Hom{\operatorname{Hom}}

\def\Ext{\operatorname{Ext}}
\def\cExt{\operatorname{\underline{Ext}}}

\def\Spec{\operatorname{Spec}}

\def\bu{\bullet}

\def\map#1{{\buildrel #1 \over \lra}} 
\def\lmap#1{{\buildrel #1 \over \longleftarrow}} 
\def\lra{\longrightarrow}
\def\into{\hookrightarrow}
\def\onto{\twoheadrightarrow}
\def\cH{\mathcal{H}}

\newcommand{\bP}{\mathbb{P}}

\newcommand{\bH}{{\mathbb{H}}}
\newcommand{\G}{\mathbb{G}}

\newcommand{\bL}{{\mathbb{L}}}

\newcommand{\F}{\mathbb{F}}
\newcommand{\E}{\mathbb{E}}
\newcommand{\bF}{\mathbb{F}}
\newcommand{\bE}{\mathbb{E}}
\newcommand{\bG}{\mathbb{G}}

\newcommand{\C}{\mathbb{C}}
\newcommand{\Z}{\mathbb{Z}}

\numberwithin{equation}{section}

\theoremstyle{plain} 
\newtheorem{thm}[equation]{Theorem}

\newtheorem{lem}[equation]{Lemma}
\newtheorem{prop}[equation]{Proposition}
\newtheorem{introthm}{Theorem}

\theoremstyle{definition}
\newtheorem{defn}[equation]{Definition}
\newtheorem{ex}[equation]{Example}

\theoremstyle{definition}

\theoremstyle{remark}
\newtheorem{rem}[equation]{Remark}
\newtheorem*{ack}{Acknowledgements}

\newcommand{\xra}[1]{\xrightarrow{#1}}
\newcommand{\xla}[1]{\xleftarrow{#1}}

\newcommand{\Perf}{\operatorname{Perf}}
\newcommand{\TP}[2]{\operatorname{TPC}(#1, #2)}

\begin{document}
\title{Matrix factorizations over projective schemes}
\author{Jesse Burke}
\address{Department of Mathematics\\ 
Universit\"at Bielefeld\\ 
33501 Bielefeld\\ 
Germany.}
\email{jburke@math.uni-bielefeld.de}
\author{Mark E. Walker}
\address{Department of Mathematics \\
University of Nebraska\\
Lincoln, NE 68588
}
\email{mwalker5@math.unl.edu}


\begin{abstract}
We study matrix factorizations of regular global sections of line
bundles on
schemes. If the line bundle is very ample relative to a Noetherian
affine scheme we
show that homomorphisms in the homotopy category of matrix
factorizations may be computed as the hypercohomology of a certain
mapping complex. Using this explicit
description, we prove an analogue of Orlov's theorem that there is a
fully faithful embedding of the homotopy category of matrix
factorizations into the singularity category of the corresponding zero
subscheme. Moreover, we give a complete description of the image of this functor.
\end{abstract}
\maketitle

\section{Introduction}
Given an element $f$ in a commutative ring $Q$, a {\em matrix
  factorization} of $f$ is a pair of $n \times n$ matrices $(A, B)$ such
that $AB = f\cdot I_n= BA$. This construction was introduced by
Eisenbud in \cite{Ei80} to study modules over the factor ring $R =
Q/(f)$. He showed that if
$Q$ is a regular local ring and $f$ is nonzero, the
minimal free resolution of every finitely generated $R$-module is
eventually determined by a matrix factorization \cite[Theorem
6.1]{Ei80}. Buchweitz observed in \cite{Bu87} (see also
\cite[3.9]{MR2101296}) that Eisenbud's Theorem implies 
that there is an equivalence
\begin{equation}\label{affine_equiv}\xymatrix@C=4em{  [MF(Q,f)] \ar[r]^(.37){\cong}_(.37){\coker} & D^b(R)/\Perf(R) =:
  \Dsing(R)} \end{equation}
between the \emph{homotopy category} of matrix
factorizations, which is defined analogously to the homotopy category of
complexes of modules, and the quotient of the bounded derived category of finitely
generated $R$-modules by perfect complexes. Recall that a complex is \emph{perfect} if it is
isomorphic in $D^b(R)$ to a bounded complex of finitely generated
projective $R$-modules.  We call $\Dsing(R)$ the
\emph{singularity category} of $R$, following
\cite{MR2101296}. The equivalence \eqref{affine_equiv} is induced by sending a matrix
factorization $(A,B)$ to the image of the $R$-module $\coker A$ in $\Dsing(R)$.

In this paper we study a scheme theoretic generalization
of matrix factorizations
by replacing $Q$ with a Noetherian separated scheme $X$,
$f$ by a global section $W$ of a line bundle $\cL$ on $X$, and
$R$ by the zero subscheme $i: Y \into X$ of
$W$. A \emph{matrix factorization} of
the  triple
  $(X, \cL, W)$ is a pair of locally free sheaves $\cE_1,
\cE_0$ on $X$ and maps
\[ \cE_1 \xra{e_1} \cE_0 \xra{e_0} \cE_1 \otimes \cL\]
such that $e_0 \circ e_1$ and $(e_1 \otimes 1_\cL) \circ e_0$ are both
multiplication by $W$.  The goal of this paper is to explore a
generalization
of the equivalence \eqref{affine_equiv} to this scheme-theoretic
setting.

The definition of matrix factorizations for schemes given here was introduced in \cite{Polishchuk:2010ys}; similar constructions have been studied in
\cite{1101.5847,1101.4051, 1102.0261}. All of these papers have
in some way dealt with generalizing \eqref{affine_equiv}. Therefore, before we
describe our contributions to this question, let us describe what is
known. The right hand side of \eqref{affine_equiv} makes sense for any
scheme, in particular the zero subscheme $Y \into X$ of $W$. For the
left hand side, one may mimic the affine case and define morphisms analogously to the
homotopy category of complexes of sheaves. We write this category by
$[MF(X, \cL, W)]_\naive$, for reasons that will be clear soon. For a
matrix factorization $(\cE_1 \xra{e_1} \cE_0 \xra{e_0}
\cE_1 \otimes \cL)$, multiplication by $W$
on $\coker e_1$ is zero, and thus we may view $\coker e_1$
as an object of $\Dsing(Y)$. There
is a functor
\begin{equation}\label{intro_funct2} [MF(X, \cL, W)]_\naive \to \Dsing(Y) \end{equation}
that sends a matrix factorization $(\cE_1 \xra{e_1} \cE_0 \xra{e_0}
\cE_1 \otimes \cL)$ to $\coker e_1$.

When $X$ is a regular scheme for which every coherent sheaf is the
quotient of a locally free sheaf, and $W$ is a regular global section of $\cL$
(i.e., $W: \cO_X \to \cL$ is injective), it is straightforward
to see that
the functor \eqref{intro_funct2} is essentially surjective. Indeed, as in the affine
case, every object of $\Dsing(Y)$ is isomorphic to a coherent
sheaf $\cM$ that is maximal Cohen-Macaulay (i.e., for each $y \in Y$, $\cM_y$
is a maximal Cohen-Macaulay (MCM) module over the ring $\cO_{Y,y}$), and one may mimic the
standard argument in the affine case that associates a matrix
factorization to a MCM module: One first takes a surjection $\cE_0
\onto i_*\cM$ with $\cE_0$ locally free on $X$. The
hypotheses ensure that  the kernel
$\cE_1$ will also be locally free. Multiplication by $W$ determines
the vertical maps in the diagram
$$
\xymatrix{
\cE_1 \ar[r]^{\alpha} \ar[d]^W & \cE_0 \ar@{-->}[ld]_\beta \ar[r]
\ar[d]^W &  i_* \cG \ar[d]^W \\
\cE_1 \otimes \cL \ar[r]_{\alpha \otimes id_{\cL}} & \cE_0 \otimes \cL \ar[r] & i_*
\cG \otimes \cL. \\
}
$$
Since the right-most such map is the zero map, there exists a
diagonal arrow $\beta$ causing both triangles to commute. We thus obtain
a pair of locally free coherent sheaves $\cE_0, \cE_1$ on $X$ and morphisms 
$$
\cE_1
\map{\alpha} \cE_0 \map{\beta} \cE_1 \otimes \cL
$$
such that both compositions $\beta \circ
\alpha$ and $(\alpha \otimes id_{\cL}) \circ \beta$ are 
multiplication by $W$. This is a matrix factorization of the data
$(X, \cL, W)$.

In general \eqref{intro_funct2} will not be an equivalence. For
observe that if the cokernel of a matrix factorization is locally
free, then it is trivial in the singularity category. 
When $X$ is affine, locally free sheaves are
projective and the lifting property of such sheaves allows one to construct a
null-homotopy.
But in the non-affine case there is no reason such a null-homotopy
should exist in general, and indeed 
there are
many examples of matrix
factorizations that are non-zero in the naive homotopy category but
that have a locally free cokernel; see  Example
\ref{non-affine-ex}. 

When $X$ is not affine one may, as is done in \cite{1101.4051, Polishchuk:2010ys},
take the Verdier quotient of the naive homotopy category by objects with a
locally free cokernel. Let us write this category as $[MF(X, \cL,
W)]$.  Then \cite[Theorem
3.14]{Polishchuk:2010ys}, see also \cite[Theorem 3.5]{1101.4051},
shows that \eqref{intro_funct2} induces an equivalence
\[[MF(X, \cL, W)] \xra{\cong} \Dsing(Y).\] In \cite{1101.4051} the
line bundle $\cL$ is assumed to be $\cO_X$ but the scheme $X$ is not
assumed to be regular. In that case it is shown that the induced
functor $[MF(X, \cL, W)] \to \Dsing(Y)$ is fully faithful. See also
\cite{1102.0261} for a proof of a similar result using exotic derived categories.

A drawback of Verdier quotients is that morphism sets in quotient
categories can be difficult to compute. In this paper
we offer a different approach to describing the category
$[MF(X, \cL, W)]$. For every
pair of matrix factorizations $\bE, \bF$ there is mapping complex,
denoted by $\HomMF( \bE, \bF)$, which is a twisted two-periodic
complex of locally free sheaves on $X$; see Definition \ref{def_hom_complex} for the precise
description. 
We define the category $[MF(X,
\cL, W)]_{\H}$ to have objects all matrix factorizations and for two
such objects $\bE, \bF$, morphisms between them are given by
\[ \Hom_{[MF]_{\H}}( \bE, \bF) = \bH^0 \HomMF( \bE, \bF),\] 
where $\bH^0$ denotes hypercohomology in degree 0. There is a
composition which is associative and unital.

Recall that a scheme $X$ is projective over a ring $Q$
if there is a closed embedding $j: X \into \bP^m_Q$ for some $m \geq
0$. In this case, we say that $\cO_X(1) := j^* \cO_{\bP^m_Q}(1)$ is the
corresponding very ample line bundle. Our first main result shows that
in this case the two homotopy categories coincide.
\begin{introthm}
\label{introthm1}
Let $X$ be a scheme that is projective over a Noetherian ring and
$\cL = \cO_X(1)$ the corresponding very ample line bundle.
For a global section $W$ of $\cL$ there is an equivalence of categories
$$
[MF(X, \cL,W)] \xra{\cong}
[MF(X, \cL,W)]_{\bH}.
$$
\end{introthm}

Using this concrete description of the morphisms in the homotopy
category, we are able to give another proof of (an analogue of) \cite[Theorem
3.4]{1101.4051} (which assumed that $\cL = \cO_X$) when $X$ is
projective over a Noetherian ring and $\cL = \cO_X(1)$.

\begin{introthm}
\label{introthm2}
Let $X$ be a scheme that is projective over a Noetherian ring of
finite Krull dimension,
$\cL = \cO_X(1)$ the corresponding very ample line bundle, and $W$ a
regular global section of $\cL$.
Define
$i: Y \into X$ to be zero subscheme of $W$. There is a 
functor
$$
\ocoker: [MF(X, \cL, W)] \to \Dsing(Y),
$$
which sends a  matrix factorization $(\cE_1 \xra{e_1} \cE_0 \xra{e_0}
\cE_1 \otimes \cL)$ to $\coker e_1$,
and which is fully faithful. The essential image is given by the objects $\cC$
in $\Dsing(Y)$ such that $i_* \cC$ is perfect on $X$. In particular, if $X$
is regular then $\ocoker$ is an equivalence.
\end{introthm}

In \cite{1101.4051} Orlov remarks that it would be
interesting to understand the difference between $[MF(X, \cL, W)]$ and
$\Dsing(Y)$. The above theorem shows that, when $X$ is projective over
an affine scheme and $\cL = \cO_X(1)$, the
difference is exactly the class of objects in $\Dsing(Y)$ that are
not perfect over $X$.

We came to these results studying affine complete intersection
rings. Let $Q$ be a regular ring,
$f_1, \dots, f_c$ a regular sequence of elements in $Q$, and set $R =
Q/(f_1, \ldots, f_c)$. Define $X = \bP^{c-1}_Q =
\Proj Q[T_1, \dots, T_c]$, $\cL = \cO_X(1)$, $W = \sum_i f_iT_i
\in \Gamma(X, \cL)$ and $Y \into X$ the zero subscheme of
$W$. Orlov showed in \cite[Theorem 2.1]{MR2437083} that there is an equivalence
\[ \Dsing(R) \xra{\cong} \Dsing(Y). \] 
Composing this with the equivalence of Theorem 2 we obtain an
equivalence 
$$
\Dsing(R)
\xra{\cong} [MF(X, \cL, W)].
$$ 
In a companion paper to this one we
will use this equivalence and the explicit description of the
Hom-sets in $[MF(X, \cL, W)]$ given here to study the
cohomology of modules over $R$.

\begin{ack}
  The authors thank Greg Stevenson for carefully reading a preliminary
  version of this paper and offering helpful comments on it.
\end{ack}

\section{The generalized category of matrix factorizations}
Throughout $X$ will denote a Noetherian separated scheme and $\cL$ a line
bundle on $X$. To simplify notation, even if $\cL$ is not very ample, for a quasi-coherent sheaf $\cG$ (or a
complex of such) on $X$ and integer $n$, we will write
$\cG(n)$ for $\cG \otimes_{\cO_X} \cL^{\otimes n}$. 
(Recall $\cL^{\otimes -n} := \uHom_{\cO_X}(\cL^{\otimes n},
\cO_X)$ for $n \geq 1$.) In particular, $\cO(1) = \cL$.
Similarly, if $f$ is morphism of (complexes
of) quasi-coherent sheaves, then $f(1) = f \otimes
id_{\cL}$.  

The following definition first appeared in \cite{Polishchuk:2010ys}.
\begin{defn}
  Let $W$ be a global section of $\cL$. A {\em matrix factorization} $\E =  (\cE_1 \xra{e_1} \cE_0 \xra{e_0}
  \cE_1(1))$ of the triple $(X, \cL, W)$ consists of a pair of
  locally free coherent sheaves $\cE_1, \cE_0$ on $X$ and morphisms $e_1: \cE_1 \to \cE_0$ and $e_0: \cE_0 \to \cE_1(1)$
  such that $e_0 \circ e_1$ and $e_1(1) \circ e_0$ are multiplication
  by $W$. A {\em strict morphism} of matrix factorizations from
  $(\cE_1 \to \cE_0  \to \cE_1(1))$ to 
  $(\cF_1 \to \cF_0  \to \cF_1(1))$ is a pair of maps $\cE_0 \to
  \cF_0, \cE_1 \to \cF_1$ causing the evident pair of squares to 
  commute. Matrix factorizations and strict morphisms of such form a category which we write
$MF(X, \cL, W)_\exact$ or just $MF_\exact$ for short.

The larger category, with objects matrix factorizations of
arbitrary coherent sheaves and arrows strict morphisms defined in
the same way as above, is an abelian category. The category 
$MF_\exact$ is a full subcategory of this abelian category and is closed under
  extensions, and hence $MF_\exact$ has the structure of  an \emph{exact
    category} in the sense of Quillen \cite{MR0338129}. A sequence $0 \to \E'
  \to \E \to \E'' \to 0$ in
  $MF_\exact$ is a short exact sequence if it determines a
 short  exact sequence of locally free coherent sheaves in both
  degrees. 
\end{defn}

\begin{defn}
\label{twisted_periodic_def} 
A \emph{twisted periodic  complex  of locally free coherent sheaves}
for $(X, \cL)$ 
is a chain complex $\cC$ of locally free coherent sheaves on $X$ together
with a specified isomorphism $\alpha: \cC[2] \xra{\cong} \cC(1)$, 
where we use the convention that $\cC[2]^i =
\cC^{i+2}$. The category $\TP X \cL$ has as objects twisted periodic
complexes and a morphism of such objects is a chain map that commutes
with the isomorphisms in the evident sense. There is an equivalence
\[ \TP X \cL \cong MF(X, \cL, 0)_\exact\]
given by sending $(\cC, \alpha)$ to $\cC^{-1} \xra{d}  \cC^0 \xra{\alpha^{-1} \circ d} \cC^{-1}(1)$.
\end{defn}

The most important example of a twisted periodic complex, for us, is
the following:
\begin{defn}
\label{def_hom_complex}
 Let $\E =  (\cE_1 \xra{e_1} \cE_0 \xra{e_0} \cE_1(1))$ and $\F =
 (\cF_1 \xra{f_1} \cF_0 \xra{f_0} \cF_1(1))$ be matrix
 factorizations for $(X, \cL, W)$. We define the {\em mapping complex} of $\bE, \bF$,
 written $\HomMF(\E, \F)$, to be the following twisted periodic complex of locally free sheaves:
\begin{equation*}
{\small
\ldots \xra{\partial^0(-1)}
\begin{matrix}
  \uHom(\cE_0, \cF_1) \\
  \oplus \\
  \uHom(\cE_1, \cF_0(-1))
\end{matrix}\xra{\partial^{-1}}
\begin{matrix}
  \uHom(\cE_0, \cF_0) \\
  \oplus \\
  \uHom(\cE_1, \cF_1)
\end{matrix}
\xra{\partial^{0}}
\left(\begin{matrix}
  \uHom(\cE_0, \cF_1) \\
  \oplus \\
  \uHom(\cE_1, \cF_0(-1))
\end{matrix}
\right)(1)
\xra{\partial^{-1}(1)} \ldots
}
\end{equation*}

Here, $\uHom$ denotes the sheaf of homomorphisms between two coherent sheaves on $X$  and $ \uHom(\cE_0, \cF_0)
  \oplus
  \uHom(\cE_1, \cF_1)$ lies in degree $0$. 
The differentials are given by
$$
\partial^{-1} = 
\begin{bmatrix}
(f_1)_* & -e_0^* \\
-e_1^* & (f_0)_* \\
\end{bmatrix}
\quad \text{ and } \quad
\partial^{0} = 
\begin{bmatrix}
(f_0)_* & e_0^* \\
e_1^* & (f_1)_* \\
\end{bmatrix},
$$
using the canonical isomorphisms 
$$
\uHom(\cE_i, \cF_j(1))  \cong
\uHom(\cE_i, \cF_j)(1)  \cong
\uHom(\cE_i(-1), \cF_j)
$$ 
and
$$
\left(\begin{matrix}
  \uHom(\cE_0, \cF_1) \\
  \oplus \\
  \uHom(\cE_1, \cF_0(-1))
\end{matrix}
\right)(1)
\cong
\begin{matrix}
  \uHom(\cE_0, \cF_1(1)) \\
  \oplus \\
  \uHom(\cE_1, \cF_0).
\end{matrix}
$$
One checks that $\partial^0 \circ \partial^{-1}$ and $\partial^{-1}(1)
\circ \partial^0$ are both $0$, and hence 
$\HomMF(\bE, \bF)$ is in fact a twisted periodic complex.
\end{defn}

Note that there is an isomorphism
\[ \Hom_{MF_\exact}(\E, \F) \cong Z^0( \Gamma(X,\HomMF(\E, \F)) )\]
where $\Gamma(X,\HomMF(\E, \F))$ is the complex of abelian groups
obtained by applying the global sections functor 
degree-wise to $\HomMF(\E, \F)$, and $Z^0$ denotes the cycles in degree 0.

\begin{defn}
A strict morphism $(g_1, g_0): \bE \to \bF$ is {\em nullhomotopic} if there are maps $s: \cE_0 \to
  \cF_1$ and $t: \cE_1(1) \to \cF_0$ as in the diagram below
\[ \xymatrix{
\cE_1 \ar[r]^{e_1} \ar[d]_{g_1}  & \cE_0 \ar@{-->}[ld]_{s} \ar[r]^{e_0}
\ar[d]^{g_0} &  \cE_1(1) \ar@{-->}[ld]_{t} \ar[d]^{g_1(1)}  \\
\cF_1 \ar[r]_{f_1} & \cF_0 \ar[r]_{f_0} & \cF_1(1) \\
}\]
such that
$$
g_1 = s\circ e_1 +
f_0(-1) \circ t(-1) \quad \text{and} \quad g_0 = f_1 \circ s + t \circ e_0.
$$
Two strict morphisms are \emph{homotopic} if their difference is nullhomotopic.

The {\em naive homotopy category} of matrix factorizations, written
$$
[MF(X, \cL, W)]_{\naive},
$$
is the category with the same objects as $MF(X, \cL,W)_\exact$ and
arrows given by strict morphisms modulo homotopy. Equivalently, for objects $\E, \F$:
$$
\Hom_{[MF]_\naive}(\bE, \bF) = H^0 \Gamma(X, \HomMF(\E, \F)).
$$ 
\end{defn}

\begin{defn}
  The \emph{shift functor} on $[MF(X, \cL, W)]_\naive$, written
  $[1]$, is the endo-functor given on objects by
$$
\left(\cE_1 \map{e_1} \cE_0 \map{e_0} \cE_1(1)\right)[1] =
\left(\cE_0 \map{-e_0} \cE_1(1) \map{-e_1(1)} \cE_0(1)\right).
$$
The \emph{cone} of a strict morphism $f = (g_1, g_0): \bE \to \bF$ is the
matrix factorization
\[ \cone(f) = \left ( \cE_0 \oplus \cF_1 \xra{\tiny{
  \left [ \! \begin{array}{rc}
    -e_0 & 0 \\
g_0 & f_1
  \end{array} \! \right ]}} \cE_1(1)
\oplus \cF_0
 \xra{ \tiny{
  \left [ \! \begin{array}{rc}
    -e_1(1) & 0 \\
g_1(1) & f_0
  \end{array} \! \right ]}
} \cE_0(1) \oplus \cF_1(1) \right ).
\]
There are maps $\bF \to \cone(f) \to \bE[1]$ defined in the usual
manner, and we define a 
{\em distinguished triangle} to be a triangle in $[MF]_\naive$
isomorphic to one of the form
\[
\bE \xra{f} \bF \to \cone(f) \to \bE[1].
\]
As remarked in \cite{Polishchuk:2010ys}, these structures make
$[MF(X, \cL,
W)]_\naive$ into a triangulated category, a fact one can check
directly by mimicking the proof for the homotopy category
of complexes in an abelian category.
\end{defn}

\begin{ex}
 Let $Q$ be a Noetherian ring, $X = \Spec Q, \, \cL = \cO_X,$ and
 $W$ an element of $Q$. Then a matrix factorization is a pair of
 projective $Q$-modules $E_1, E_0$ and maps 
\[ \xymatrix{ E_1 \ar@<.5ex>[r] & \ar@<.5ex>[l] E_0}\]  such that composition in either direction is multiplication by $W$.
The category $[MF(X, \cL,
W)]_{\naive}$ is the \emph{homotopy category of matrix
  factorizations}, as defined for instance in \cite[3.1]{MR2101296},
where it is denoted $HMF(Q, W)$.
\end{ex}

For a point $x \in X$, we may localize an object $\E$ of
  $MF(X, \cL, W)_\exact$ at $x$ in the evident manner to obtain an object of
  $MF(\cO_{X,x}, \cL_x, W_x)_\exact$. It
is clear that the functor
$$
MF(X, \cL, W)_\exact \to MF(\cO_{X,x}, \cL_x, W_x)_\exact, \, \bE \mapsto \bE_x
$$
preserves homotopies, and in this way we obtain a triangulated functor
$$
[MF(X, \cL, W)]_\naive \to [MF(\cO_{X,x}, \cL_x, W_x)]_\naive.
$$

\begin{defn}
A map of matrix factorizations $\E \to \E'$ in $[MF(X, \cL,
W)]_\naive$
is a {\em weak
  equivalence}  if for each $x \in X$, the map $\E_x \to \E'_x$ is an
isomorphism in the category $[MF(\cO_{X,x}, \cL_x,
W_x)]_\naive$.

A matrix factorization $\F$ is {\em locally contractible} if the
unique map $0 \to \F$ is a weak
equivalence. This is equivalent to the condition that 
$\F_x$ is contractible for all $x \in X$.
\end{defn}

\begin{rem}
It is asserted in \cite[2.6]{1101.5847} that there is a model category
structure for the evident generalization of the notion of matrix
factorizations for $(X, \cL, W)$ 
in which the objects involve a pair of arbitrary quasi-coherent
sheaves. This larger category, which is closed under all small
coproducts, plays an important role in the work of Positselski
\cite{1102.0261} and Lin and Pomerleano \cite{1101.5847}.
\end{rem}

For a global section $W$ of a line bundle $\cL$, the \emph{zero
  subscheme of $W$} is the subscheme of $X$ determined by the ideal sheaf given as
the image of $W^\vee:
\cL^\vee \to \cO_X$.
\begin{prop}
\label{affine_loc_acyclic}
Assume $X$ is a Noetherian scheme, $\cL$ is a line bundle on $X$, and
$W$ is a
regular global section of $\cL$, i.e. the map $W: \cO_X \to \cL$ is
injective. Let $i: Y \into X$ be the zero subscheme of $W$. Consider the following conditions on a matrix factorization
$\bE = (\cE_1 \xra{e_1} \cE_0 \xra{e_0} \cE_1(1)):$
\begin{enumerate}
\item $\bE \cong 0$ in the category
$[MF(X, \cL, W)]_{\naive}$.

\item The canonical surjection $p: i^* \cE_0 \onto i^* \coker(e_1)$ of
  coherent $\cO_Y$-sheaves splits.

\item $\bE$ is locally contractible.

\item $i^* \coker(e_1)$ is a locally free coherent sheaf on $Y$.

\end{enumerate}
In general, we have $(1) \Rightarrow (2)\Rightarrow (3) \Longleftrightarrow (4)$. If $X$ is affine then all four conditions are
equivalent.
\end{prop}

\begin{rem} Note that multiplication by $W$ annihilates $\coker(e_1)$, and
  hence $i_* i^* \coker(e_1) \cong \coker(e_1)$ via the canonical
  map.
\end{rem}
 
\begin{proof}
We first observe that the proof of \cite[3.8]{MR2101296} applies to show
that (4) implies (1) when $X$ is affine.

We next prove that (1) implies (2).
Let $\cM = i^* \coker(e_1) = \coker(i^* e_1)$.
If $\bE \cong 0$, then $\bE$ has a contracting homotopy, given by $s:
\cE_0 \to \cE_1$ and  $t: \cE_1(1) \to \cE_0$ satisfying $e_1s + te_0
= id$ and $s(1)e_1(1) + e_0t = id$.
As in
the argument in the proof of \cite[3.7]{MR2101296}, since $i^*(t e_0 e_1) =
0$ and $\cM$ is the cokernel of $i^*(e_1)$, there is a map $j: \cM \to
i^*\cE_0$ such that $j p = i^*(te_0)$. But then
$$
pjp = pi^*(te_0) =  pi^*(id - e_1s) = p
$$
since $p i^*(e_1) = 0$. Since $p$ is onto, we have $pj = id_\cM$.

If $(i^* \coker \bE)_x$ is a free
$\cO_{Y,x}$ module for all $x \in Y$, then, since (4) implies (1) in
the affine case, we conclude that $\bE_x = 0 \in [MF(\cO_{X,x}, \cL_X,
W_x)]$. This proves that (4) implies
(3). It is clear that (2) implies (4) and hence that (2) implies (3).

Finally, if $\bE_x \cong 0$, then since (1) implies (2), we see that the map
$p_x: (i^* \cE_0)_x \onto (i^* \coker(e_1))_x$ splits and hence $(i^*
\coker(e_1))_x$ is free. This proves (3) implies (4).
\end{proof}

The following shows that a locally contractible matrix
factorization need not be $0$ in $[MF]_\naive$.

\begin{ex}
\label{non-affine-ex}
Under the assumption of Proposition \ref{affine_loc_acyclic},
suppose that $\bE = (\cE_1 \to \cE_0 \to \cE_1(1))$ is
a matrix factorization of $(X, \cL, W)$ such that $i^* \cM :=
i^* \coker(\cE_1 \to \cE_0)$ is a locally
free coherent $\cO_Y$-sheaf,
but that the surjection $i^* \cE_0 \onto i^* \cM$ does not
split.
Then by Proposition \ref{affine_loc_acyclic}, 
$\bE$ is locally contractible, but
$\bE$ is not isomorphic
to $0$ in $[MF]_\naive$.
Such examples are common. For instance, take $X =
\bP^2_k = \Proj k[T_0, T_1, T_2]$, $\cL = \cO_X(1)$ and $W = T_2$,  so
that $Y = \bP^1_k = \Proj k[T_0, T_1]$. Then let
$\cM = \cO_Y$, $\cE_0 = \cO_X(-1)^2$ and $\cE_0 \onto i_* \cM$ be the
composition of
$$
\cE_0 \xra{\text{can}} i_* \cO_Y(-1)^2 \xra{(T_0, T_2)} i_*\cO_Y.
$$
Then the kernel $\cE_1$ of $\cE_0 \onto i_* \cM$ is locally free and,
using the argument found in the introduction, this leads to a matrix factorization of $(X,
\cL, W)$ with $i^* \coker(\cE_1 \to \cE_0) = \cM$.
But the surjection
$$
i^* \cE_0  = \cO_Y(-1)^2 \xra{(T_0, T_1)} \cO_Y = \cM
$$
does not split.
\end{ex} 

The collection of locally contractible objects
is the intersection
of the kernels of 
the triangulated functors
$$
[MF(X, \cL, W)]_\naive
\to
[MF(\cO_{X,x}, \cL_x, W_x)],
$$
as $x$ ranges over all points of $X$. Recall that a triangulated subcategory
of a triangulated category is thick if it is closed under direct
summands. The kernel of any triangulated functor is thick and
an arbitrary intersection of thick subcategories is thick. Thus
the collection of locally contractible objects forms a thick subcategory of $[MF(X,
\cL, W)]_\naive$.

The following category, whose definition is originally due to Orlov
and appeared, for example,  in \cite{Polishchuk:2010ys},
is the central object of study in this paper:
\begin{defn} The \emph{homotopy category of matrix factorizations},
  written $[MF(X, \cL, W)]$, is the Verdier
  quotient of $[MF(X, \cL, W)]_\naive$ by the thick subcategory of
  locally contractible objects:
$$
[MF(X, \cL, W)] = \frac{[MF(X, \cL, W)]_\naive}{\text{locally
    contractible objects}}.
$$
\end{defn}

\begin{rem}
  A strict map $\E' \to \E$ is a weak equivalence if and only if it
  fits into a distinguished triangle
$$
\E' \to \E \to \F \to \E'[1]
$$
in $[MF(X, \cL, W)]_\naive$ such that $\F$ is locally
contractible. Thus weak equivalences are invertible in $[MF(X, \cL,
W)]$.
\end{rem}

\begin{ex} \label{L713}
If $X$ is affine, $\cL = \cO_X$  and $W$ is a
non-zero-divisor, 
then $[MF(X, \cL, W)] = [MF(X,
\cL, W)]_\naive$ by Proposition \ref{affine_loc_acyclic}.
\end{ex}

\section{Another version of the homotopy category}
\label{another_homotopy}
The aim of this section is to describe another triangulated
category associated to $MF(X, \cL, W)$, which we write as
$[MF(X, \cL ,W)]_\bH$.  
In the next section, we prove that when $X$ is projective over
a Noetherian ring and $\cL$ is the corresponding very ample line bundle,
$[MF]_\bH$ and $[MF]$ are equivalent. The advantage $[MF]_\bH$ enjoys over $[MF]$ is that its
$\Hom$ sets are more explicit.

We make a fixed choice of a
  finite affine open cover $\cU = \{U_1, \dots, U_m\}$ of $X$, and for
  any quasi-coherent sheaf $\cF$ on $X$, let $\Gamma(\cU, \cF)$ denote
  the cochain complex given by the usual Cech construction. Since $X$ is separated, the cohomology
  of the complex
  $\Gamma(\cU, \cF)$ gives the sheaf cohomology of $\cF$.
We  define
 $\Gamma(\cU, \HomMF(\E, \F))$ to be
 the total complex associated to the bicomplex
$$
0 \to \bigoplus_i \Gamma(U_i,\HomMF(\E,
  \F)) \to \bigoplus_{i<j} \Gamma(U_i \cap U_i,\HomMF(\E, \F)) \to \cdots
$$
given by applying the Cech construction degree-wise.
If $\G$ is another matrix factorization, there is an evident morphism
of chain complexes
$$
\Gamma(\cU, \HomMF(\E, \F)) \otimes \Gamma(\cU, \HomMF(\F,
\G)) \to \Gamma(\cU, \HomMF(\E, \G))
$$
which one can check is associative and unital. Thus $MF(X, \cL, W)$,
with function spaces
$\Gamma(\cU, \HomMF(\E,\F))$, is a DG category.
We set \[ \bH^q(X, \HomMF(\E, \F)) = H^q( \Gamma(\cU, \HomMF(\E, \F))
). \]

There is a convergent spectral
sequence
$$
H^p(X, \cH^q(\HomMF(\E, \F))) \Longrightarrow \bH^{p+q}(X, \HomMF(\E, \F))
$$
where $\cH^q$ is the $q$-th cohomology sheaf of a complex.
In particular, 
if $\HomMF(\E, \F) \to \HomMF(\E', \F')$ is a 
quasi-isomorphism,
then the map
$$
\bH^n(X, \HomMF(\E, \F)) \to \bH^n(X, \HomMF(\E', \F'))
$$
is an isomorphism for all $n$.

\begin{defn}
Define the category
$$
[MF(X, \cL, W)]_{\bH}
$$
whose objects are matrix factorizations
and whose morphisms are
$$
\Hom_{[MF]_{\bH}}(\E, \F) = \bH^0(X, \HomMF(\E, \F)).
$$
Thus $[MF]_{\bH}$ is the homotopy category associated to the
DG category above.
\end{defn}

\begin{rem}
 This definition was inspired by Shipman's
 category of \emph{graded D-Branes} \cite{1012.5282}, who was in turn
 inspired by Segal's \cite{Segal}.
\end{rem} 

There is a canonical functor
\begin{equation} \label{E815}
[MF(X, \cL, W)]_\naive \to [MF(X, \cL, W)]_{\bH}
\end{equation}
that is the identity on objects and is given on morphisms by the
canonical map
$$
H^0 \Gamma(X, \HomMF(\bE, \bF)) \to \bH^0(X, \HomMF(\bE, \bF)).
$$

\begin{ex}
If $X$ is affine this
  functor is an equivalence since each map \[H^0 \Gamma(X, \HomMF(\bE,
  \bF)) \to \bH^0(X, \HomMF(\bE, \bF))\] is an isomorphism.  If we
  further assume $\cL = \cO_X$ and $W$ is a non-zero-divisor of $Q$,
then both
$[MF(X, \cL, W)]_\naive$ and $[MF(X, \cL, W)]_{\bH}$
  are equivalent to 
$[MF(X, \cL, W)]$ by Example \ref{L713}.
\end{ex}

\begin{lem} \label{Lem1}
If a strict morphism $f: \bE' \to \bE$ of matrix
  factorizations is a weak equivalence, then 
for all matrix factorizations $\bF$, the induced
map on mapping complexes
$$
\HomMF(\bE, \bF) \to
\HomMF(\bE', \bF)
$$ 
is a quasi-isomorphism in $TPC(X, \cL)$. In particular, the map in $[MF(X, \cL, W)]_{\bH}$ induced by $f: \bE' \to \bE$ is an isomorphism.
\end{lem}

\begin{proof}
For matrix factorizations $\E, \F$ and for all $x \in X$ there is an isomorphism
$$
\uHom_{MF(X, \cL, W)}(\E, \F)_x \cong
\uHom_{MF(\cO_{X,x}, \cL_x, W_x)}(\E_x, \F_x).
$$
Since we are assuming $\E'_x \to \E_x$ is an isomorphism in
$[MF(\cO_{X,x}, \cL_x, W_x)]$, it follows that
$$
\uHom_{MF(X, \cL, W)}(\E, \F)_x 
\to
\uHom_{MF(X, \cL, W)}(\E', \F)_x 
$$
is a quasi-isomorphism for all $\F$ and $x$. This proves 
$$
\HomMF(\bE, \bF) \to
\HomMF(\bE', \bF)
$$ 
is a quasi-isomorphism in $TPC(X, \cL)$  and hence that
$$
\bH^n(X, \HomMF(\bE, \bF)) \map{\cong}
\bH^n(X, \HomMF(\bE', \bF))
$$
is an isomorphism for all $n$. The case $n=0$ shows $\bE$ and $\bE'$ co-represent
the same functor on $[MF(X, \cL, W)]_{\bH}$ and hence are isomorphic. 
\end{proof}

The following is a formal consequence of the lemma:

\begin{prop} The functor 
$[MF(X, \cL, W)]_\naive \to [MF(X, \cL, W)]_{\bH}$ factors canonically as
$$
[MF(X, \cL, W)]_\naive \to [MF(X, \cL, W)] 
\to [MF(X, \cL, W)]_{\bH}.
$$
\end{prop}

\section{Equivalence of homotopy categories}
The goal of this section is to prove Theorem \ref{introthm1}
of the introduction: the functor
$$
[MF(X, \cL, W)] \to [MF(X,
\cL, W)]_{\bH}$$ 
is an equivalence of categories when $X$ is projective over a
Noetherian ring and $\cL$ is the corresponding very ample line bundle. For such an $X$ and $\cL$, we will often use
Serre's Vanishing Theorem: for any coherent sheaf $\cF$ on $X$ and
all $n \gg 0$, $H^i(X, \cF(n) ) = 0$ for $i > 0$.

Given a bounded complex of strict morphisms of matrix factorizations
$$
\E^\bu := \left(\E^p\to \bE^{p+1} \to \cdots \to \E^q\right),
$$
we define its {\em total object}, written $\Tot(\E^\bu)$, as follows. 
We may visualize $\bE^\bu$ as a commutative diagram
$$
{\small \xymatrix{ \cdot & \cdots \ar[r] & \cE_{0}^q(-1) \ar[r] &
    \cE_{1}^q \ar[r] & \cE_{0}^q
    \ar[r] & \cE_{1}^q(1) \ar[r] & \cdots \\
    & \cdots \ar[r] & \cE^{q-1}_{0}(-1) \ar[r] \ar[u] & \cE^{q-1}_{1}
    \ar[r] \ar[u] & \cE^{q-1}_{0}
    \ar[r]  \ar[u] & \cE^{q-1}_{1}(1) \ar[r]  \ar[u] & \cdots \\
    \cdot \ar[uu]^0& \cdots \ar[r] & \cE^{q-2}_{0}(-1) \ar[r] \ar[u] &
    \cE^{q-2}_{1}\ar[r] \ar[u] & \cE^{q-2}_{0}
    \ar[r]  \ar[u] & \cE^{q-2}_{1}(1) \ar[r]  \ar[u] & \cdots \\
    & & \vdots \ar[u] & \vdots \ar[u] & \vdots \ar[u] &\vdots \ar[u] & \\
    & \cdots \ar[r] & \cE^{p}_{0}(-1) \ar[r] \ar[u] & \cE^{p}_{1}\ar[r]
    \ar[u] & \cE^p_0
    \ar[r]  \ar[u] & \cE^p_1(1) \ar[r]  \ar[u] & \cdots \\
    &&& \cdot \ar[rr]^{W} && \cdot \\
  } }.
$$
We form
$\Tot(\bE^\bu)$ by taking direct sums along lines of slope -1 in this
diagram, and maps are defined just as for the usual total complex
associated to a bicomplex.  The resulting chain of maps clearly
satisfies the required twisted periodicity, making it an object of
$MF(X, \cL, W)$.

A special case of the $\Tot$ construction will be especially
useful. First, for a locally free coherent sheaf $\cP$ on $X$ and a matrix
factorization $\bE$, let $\cP \otimes \bE$ denote the matrix
factorization obtained by applying the functor $\cP \otimes_{\cO_X} -$
to the data defining $\bE$.  If $\cP^\bu$ is a bounded complex of
locally free coherent sheaves on $X$, then
$$
\cP^\bu \otimes \bE = \left(\cP^p \otimes \bE \to \cP^{p+1} \otimes
  \bE \to \cdots \to \cP^q \otimes \bE\right)
$$
is a bounded complex of strict morphisms of matrix factorizations, and we may form
its associated total matrix factorization $\Tot(\cP^\bu \otimes \bE)$.

We will need the following lemma.

\newcommand{\cQ}{\mathcal{Q}}
\begin{lem} \label{Lem3} Let $X$ be a Noetherian separated scheme,
  $\cL$ any line bundle on $X$ and $W$ a global section of $\cL$.
If $\cP^\bu_1 \to \cP^\bu_2$ is a quasi-isomorphism
between bounded complexes of locally free coherent sheaves on $X$, then 
$$
\Tot(\cP^\bu_1 \otimes \E) \to
\Tot(\cP^\bu_2 \otimes \E)
$$
is a weak equivalence of matrix factorizations.
\end{lem}

\begin{proof}  We may localize at a point $x$, in which case the
  assertion becomes: 
If $W$ is an element in a local ring
  $Q$, $P^\bu_1 \to P^\bu_2$ is a
  quasi-isomorphism of bounded complexes of free $Q$-modules of finite
  rank, and
  $\bE$ is a matrix factorization of $W$ over $Q$, then 
$$
\Tot(P^\bu_1 \otimes \bE) \to
\Tot(P^\bu_2 \otimes \bE) 
$$
is an isomorphism in the category $[MF(\Spec Q, \cO, W)]$.
In this setting, $P^\bu_1 \to P^\bu_2$ is a chain homotopy equivalence
(i.e., there is an inverse up to chain homotopy). It therefore
suffices to note that the functor
$\Tot(- \otimes \bE)$ from bounded complex of free modules to matrix
factorizations sends chain homotopies of complexes to homotopies of
matrix factorizations.
\end{proof}

\begin{thm} 
\label{MainThm1}
Let $X$ be a scheme that is projective over a Noetherian ring $Q$ and
$\cL = \cO_X(1)$ the corresponding very ample line
bundle on $X$. 
For any global section $W$ of $\cL$, the canonical functor
$$
[MF(X, \cL,W)] \to
[MF(X, \cL,W)]_{\bH}
$$
is an equivalence. 
\end{thm}

The proof of the theorem uses the following lemma.

\begin{lem} \label{Lem2}
Under the assumptions of Theorem \ref{MainThm1}, for any pair of matrix factorizations $\bE, \bF$, there is a weak
equivalence $\bE' \to \bE$ such that the canonical map
$$
\Hom_{[MF]_\naive}(\bE', \bF) \map{\cong}
\Hom_{[MF]_{\bH}}(\bE', \bF)
$$
is an isomorphism.
\end{lem}

\begin{proof} 
Say $X$ is a closed subscheme of $\bP^{m}_Q = \Proj Q[x_0, \dots, x_m]$
and $\cL = \cO_X(1)$ is the restriction of $\cO_{\bP^{m} }(1)$ to $X$.

For each positive integer $j$ we
have a surjection
$$
\cO_X(-j)^k \onto \cO_X
$$
given by the the set of monomials of degree $j$ in $m+1$ variables (so
that $k = k(j)$ is the number of such monomials). Let $\cP(j)^\bu$ be
the associated ``truncated'' Koszul complex
$$
0 \to \cO_X(-kj)^{k \choose k} \to \cO_X(-(k-1)j)^{k \choose k-1}
\to \cdots \to
\cO_X(-2j)^{k \choose 2}
\to \cO_X(-j)^k \to 0,
$$
indexed cohomologically with
$\cO_X(-nj)^{k \choose n}$ in degree $-n+1$.
Then the evident map 
$$
\cP(j)^\bu \map{\sim} \cO_X
$$
is a quasi-isomorphism, and so 
by Lemma \ref{Lem3}, the induced map
$$
\bE' := \Tot(\cP(j)^\bu \otimes \E ) \map{\sim} \Tot(\cO_X \otimes \bE) = \bE
$$
is a  weak equivalence. 
We prove that for $j \gg 0$ this weak equivalence 
has the desired property.

Write
$\cE_0, \cE_1$ and $\cE'_0, \cE'_1$ for the components of $\bE$ and
$\bE'$.
We have that
$$
\cE'_0
= \cE_0(-j)^k \oplus \cE_1(1-2j)^{k \choose 2} \oplus \cE_0(1-3j) ^{k
  \choose 3} \oplus \cE_1(2-4j) ^{k
  \choose 4}
\oplus \cdots
$$
and similarly for $\cE'_1$. In particular, for any integer
$N$ we may choose $j \gg 0$ so that 
$\cE'_0$ and $\cE'_1$ are direct
sums of locally free coherent sheaves of the form $\cE_0(-a)$ and $\cE_1(-b)$ with $a,b
\geq N$.

Let $\cC$ denote the twisted two-periodic complex $\HomMF(\bE',
\bF)$. By the above, for any integer $N$, we may pick $j \gg 0$ so that
$\cC^0$ and $\cC^1$ are direct sums of locally free coherent sheaves of the form
\begin{equation} \label{L415b}
\uHom_{\cO_X}(\cE_0,\cF_0)(a),
\uHom_{\cO_X}(\cE_1,\cF_0)(b),
\uHom_{\cO_X}(\cE_0,\cF_1)(c), \, \text{ or } \,
\uHom_{\cO_X}(\cE_1,\cF_1)(d)
\end{equation}
with $a,b,c,d \geq N$. Thinking of $\cC$ as an unbounded complex we have
$\cC^q = 
\cC^0(\frac{q}{2})$ if $q$ is even or 
$\cC^q = 
\cC^1(\frac{q-1}{2})$ if $q$ is odd. Thus, for any integers $N$ and
$M$, we may choose $j$ sufficiently large so that
each of $\cC^M, \cC^{M+1}, \cdots $ is a direct sum of 
locally free coherent sheaves as in \eqref{L415b} with $a,b,c,d \geq N$. In
particular, for any $M$, we may choose $j$ sufficiently large so that
$H^p(X, \cC^q) = 0$ for all $p >0$ and all $q \geq M$.

The result now follows from the spectral sequence
$$
E_1^{p,q} = H^p(X, \cC^q) \Longrightarrow \bH^{p+q}(X, \HomMF(\bE', \bF))
$$
using that $X$ has bounded cohomological dimension for quasi-coherent sheaves.  In more detail,
if $X$  has cohomological dimension $n$, we may 
choose $j$ sufficiently large so that 
$H^p(X, \cC^q) = 0$ for all $p > 0$ and $q \geq -n-1$. It follows that 
$$
\bH^{0}(X, \HomMF(\bE', \bF))
\cong E_2^{0,0} = H^0 \Gamma(X,\HomMF(\bE', \bF)). \qedhere
$$
\end{proof}

\begin{rem} The proof of the Lemma actually shows that for any matrix
  factorizations $\E, \F$ and any integer
  $M$, there is a weak equivalence $\E' \to \E$ so that
$$
H^q \Gamma(X, \HomMF(\E', \F)) \to
\bH^q (X, \HomMF(\E', \F)) 
$$
is an isomorphism for all $q \geq M$.
\end{rem}

\begin{proof}[Proof of Theorem \ref{MainThm1}]
We need to prove 
\begin{equation} \label{L415}
\Hom_{[MF]}(\bE, \bF) \to
\Hom_{[MF]_{\bH}}(\bE, \bF) 
\end{equation}
is an isomorphism for every
pair of matrix factorizations $\bE$,
$\bF$. Given such a pair, let $\bE' \to \bE$ be a weak equivalence as in Lemma \ref{Lem2}
so that
$$
\Hom_{[MF]_\naive}(\bE', \bF) \map{\cong}
\Hom_{[MF]_{\bH}}(\bE', \bF),
$$
and consider the commutative diagram
$$
\xymatrix{
\Hom_{[MF]_\naive}(\E, \F)
\ar[r] \ar[d] &
\Hom_{[MF]}(\E, \F)
\ar[r] \ar[d]^\cong &
\Hom_{[MF]_{\bH}}(\E, \F) \ar[d]^\cong \\
\Hom_{[MF]_\naive}(\E', \F)
\ar[r] &
\Hom_{[MF]}(\E', \F)
\ar[r] &
\Hom_{[MF]_{\bH}}(\E', \F). \\
}
$$
The middle and right-most vertical arrows are isomorphisms and the
composition of the arrows along the bottom is an isomorphism. It
follows that \eqref{L415} is onto.

Suppose $\alpha \in \Hom_{[MF]}(\E, \F)$ is in the
kernel of \eqref{L415}. We may represent $\alpha$ by a diagram
of strict morphisms
$$
\E \lmap{s} \G \map{\beta} \F
$$
with $s$ a weak equivalence. Let $\G' \to \G$ be a weak equivalence
for the pair $\G, \F$ given by Lemma \ref{Lem2}, so that
$$
\Hom_{[MF]_\naive}(\bG', \bF) \map{\cong}
\Hom_{[MF]_{\bH}}(\bG', \bF)
$$
is an isomorphism. By precomposing the above diagram for $\alpha$ with
the weak equivalence $\G' \to \G$, we may represent $\alpha$ also as a
diagram of strict morphisms of
the form 
$$
\E \lmap{s'} \G' \map{\beta'} \F
$$
with $s'$ a weak equivalence.
Since $\alpha$ is mapped to zero and $s'$ is mapped to an isomorphism,
$\beta'$ is mapped to zero in $[MF]_{\bH}$. But then since
$$
\Hom_{[MF]_\naive}(\bG', \bF) \map{\cong}
\Hom_{[MF]_{\bH}}(\bG', \bF)
$$
is an isomorphism, we have that
$\beta'$ is zero already in $[MF]_\naive$ and hence also
is zero in $[MF]$. It follows that $\alpha = 0$ in
$[MF]$.  
\end{proof}

\section{Hom-sets in the singularity category}
In this section we return to the general situation, with $X$ a Noetherian separated
scheme and $\cL$  a line bundle. But we make the further assumption
that  $W \in \Gamma(X, \cL)$ is a {\em regular}
section, i.e. $W: \cO_X \to \cL$ is injective. Equivalently, for each $x  \in X$, the element $W_x \in
\cO_{X,x} \cong \cL_x$ is a 
non-zero-divisor. 
Define $Y$ to be
the zero subscheme of $W$ (i.e., by the ideal given
as the image of the injective map $W^*: \cL^* \to \cO_X$).

The {\em singularity category} of a 
scheme $Z$ is the Verdier quotient
$$
\Dsing(Z) := D^b(Z)/\Perf(Z),
$$
where $D^b(Z)$ is the bounded derived category of coherent
sheaves and $\Perf(Z)$ is the full subcategory consisting
of perfect complexes --- i.e., those complexes that are locally
quasi-isomorphic to bounded complexes of free modules of finite rank. This construction was introduced by Buchweitz \cite{Bu87}
in the case when $Z$ is affine
and rediscovered by Orlov \cite{MR2101296}.

For a matrix factorization
$$
\bE  = \left(\cE_1 \map{e_1} \cE_0 
  \map{e_0} \cE_1(1)\right)
$$ 
define $\coker(\bE)$ to be $\coker(e_1)$. Multiplication
by $W$ on the coherent sheaf $\coker (\bE)$ is zero and hence $\coker(\bE)$ may be
regarded as a coherent sheaf on $Y$ and thus as an object of
$D_{sing}(Y)$. (More formally, the canonical map gives an isomorphism $\coker(e_1) \cong i_*i^*
\coker(e_1)$ and we
define $\coker(\bE) = i^* \coker(e_1)$.)

\begin{defn}
For a matrix factorization $\E =  (\cE_1 \xra{e_1} \cE_0 \xra{e_0}
\cE_1(1))$, $i^*\bE$ is the chain complex of locally free coherent
sheaves on $Y$
$$
\cdots \to  i^*\cE_0(-1) \map{e_0(-1)} 
i^*\cE_1 \map{e_1} i^*\cE_0 \map{e_0}
i^*\cE_1(1) \map{e_1(1)}  i^*\cE_0(1) \to \cdots.
$$
\end{defn}

We write $\uHom_{\cO_Y}(\cM,\cN)$ for the sheaf of homomorphisms
between two sheaves $\cM, \cN$ on $Y$ (or the total product complex of
the bicomplex of such if $\cM$ or $\cN$ is a complex) and
$\cExt^*_{\cO_Y}(\cM,\cN)$ for the
right derived functors of $\uHom(\cM, -)$.

\begin{lem}
\label{sheaf_ext} Assume $X$ is a Noetherian separated scheme, $\cL$
is a
line bundle on $X$, and $W$ is a regular global section of $\cL$. 
Let $\bE$ be a matrix factorization of $(X, \cL, W)$, set $\cM = \coker \bE$, and let $\cN$ be any coherent sheaf on $Y.$ The
 following hold:
\begin{enumerate}
\item \label{loc_free_res} $i^* \bE$ is an acyclic complex and the brutal truncation $(i^* \bE)_{\geq 0}$ is a resolution of $\cM$ by locally free
coherent sheaves on $Y$.

\item For all $q \geq 1$, there is an isomorphism \[\cExt^q_{\cO_Y}(\cM ,\cN)(1) \cong \cExt^{q+2}_{\cO_Y}(\cM,\cN).\]

\item If $\cN = \coker(\bF)$ for a matrix factorization $\bF$,
  then there is a quasi-isomorphism, natural in both $\bE$ and $\bF$,
$$
\HomMF({\bE}, {\bF} ) \to i_* \uHom_{\cO_Y}( i^* \bE, \cN ).
$$

\item If $X$ is projective over a ring and $\cL$ is
  the corresponding very ample line bundle, then the edge map of the
  local-to-global spectral sequence 
\[ \Ext^q_{\cO_Y}(\cM,\cN) \to \Gamma(Y, \cExt^q_{\cO_Y}(\cM ,\cN)) \]
is an isomorphism for $q \gg 0$. \label{edge_map_ss}
(Here $\cN$ can be an arbitrary coherent sheaf.)
\end{enumerate}
\end{lem}

\begin{proof}
It is clear that $i^*\bE$ is a complex, since the composition of two
adjacent maps is multiplication by the image of $W$ in $\cO_Y(1)$,
which is zero. Acyclicity may be checked locally, and for any point
$x \in X$, the complex $(i^* \bE)_x$ is a two-periodic complex
induced by a matrix factorization over $\cO_{X,x}$ of a non-zero divisor, and hence is
acyclic by \cite[5.1]{Ei80}.
The rest of part (1) now follows directly since $\cM$ was defined to be
$i^* \coker(e_1)$.

For part (2), there is an isomorphism $i^*\bE(-1) \cong i^*
\bE[-2]$, and hence, using part (1), we have
$$
\begin{aligned}
\cExt^q_{\cO_Y}(\cM & ,\cN) (1)  \cong \cH^q( \uHom_{\cO_Y}(i^*
  \bE, {\cN}) )(1) \cong \cH^q( \uHom_{\cO_Y}(i^* \bE(-1), {\cN}) ) \\
& \cong \cH^q( \uHom_{\cO_Y}(i^* \bE[-2], {\cN}) ) 
  \cong \cH^{q+2}( \uHom_{\cO_Y}(i^* \bE, {\cN}) )
\cong \cExt^{q+2}_{\cO_Y}(\cM ,\cN),
\end{aligned}
$$
for $q \geq 1$.

For part (3), 
there is a surjective map of chain complexes,
\begin{equation} \label{E813c}
\HomMF(\bE, \bF) \onto \uHom_{\cO_X}(\bE, i_* \cN),
\end{equation}
which is natural in both variables,
given by 
the diagram
$$
{\small
\xymatrix{
\cdots \ar[r]
& 
{\begin{matrix}
  \uHom(\cE_0, \cF_1) \\
  \oplus \\
  \uHom(\cE_1, \cF_0)(-1)
\end{matrix} }
\ar[r]^{\partial^{-1}} \ar@{->>}[d]^{(0,p_*)} &
{\begin{matrix}
  \uHom(\cE_0, \cF_0) \\
  \oplus \\
  \uHom(\cE_1, \cF_1)
\end{matrix}}
\ar[r]^{\partial^{0}} \ar@{->>}[d]^{(p_*,0)} &
{\begin{matrix}
  \uHom(\cE_0, \cF_1)(1) \\
  \oplus \\
  \uHom(\cE_1, \cF_0)
\end{matrix}}
\ar[r]^{\phantom{XXXX} h^{-1}(1)} \ar@{->>}[d]^{(0,p_*)} &
\cdots \\
\cdots \ar[r] &  \uHom(\cE_1, i_* \cN)(-1) \ar[r]^{-e_0^*} &
\uHom(\cE_0, i_* \cN) \ar[r]^{e_1^*} & 
\uHom(\cE_1, i_* \cN) \ar[r]^{\phantom{XXXX}-e_0^*(1)} &
\cdots, \\
}}
$$
where $p: \cF_0 \onto i_* \cN$ is the canonical map.
(Recall that
$$
\partial^{-1} = 
\begin{bmatrix}
(f_1)_* & -e_0^* \\
-e_1^* & (f_0)_* \\
\end{bmatrix}
\quad \text{ and } \quad
\partial^{0} = 
\begin{bmatrix}
(f_0)_* & e_0^* \\
e_1^* & (f_1)_* \\
\end{bmatrix}.)
$$
Since 
$\uHom_{\cO_X}(\bE, i_* \cN)$ is canonically isomorphic to 
$i_* \uHom_{\cO_Y}(i^*\bE, \cN)$, it suffices to prove \eqref{E813c} is
a quasi-isomorphism. 
Since $\uHom_{\cO_X}(\cE_0, -)$ and $\uHom_{\cO_X}(\cE_1,-)$ are
exact functors and $0 \to \cF_1 \to \cF_0 \to i_* \cN \to 0$ is an exact
sequence, the kernel of \eqref{E813c} is
$$
\cdots \to
\begin{matrix}
\uHom(\cE_0, \cF_1) \\ \oplus \\
\uHom(\cE_1, \cF_1) (-1) 
\end{matrix}
\xra{
{\tiny
\begin{bmatrix}
1 & -e_0^* \\
-e_1^* & 0 \\
\end{bmatrix}
}}
\begin{matrix}
\uHom(\cE_0, \cF_1) \\  \oplus \\
\uHom(\cE_1, \cF_1) 
\end{matrix}
\xra{{\tiny
\begin{bmatrix} 
0 & e_0^* \\
e_1^* & 1 \\
\end{bmatrix}
}}
\begin{matrix}
\uHom(\cE_0, \cF_1)(-1) \\ \oplus \\
\uHom(\cE_1, \cF_1)  
\end{matrix}
\to \cdots.
$$
The maps
$$
\cdots 
\leftarrow
\begin{matrix}
\uHom(\cE_0, \cF_1) \\ \oplus \\
\uHom(\cE_1, \cF_1) (-1) 
\end{matrix}
\xla{
{\tiny
\begin{bmatrix}
1 & 0 \\
0 & 0 \\
\end{bmatrix}
}}
\begin{matrix}
\uHom(\cE_0, \cF_1) \\  \oplus \\
\uHom(\cE_1, \cF_1) 
\end{matrix}
\xla{{\tiny
\begin{bmatrix} 
0 & 0 \\
0 & 1 \\
\end{bmatrix}
}}
\begin{matrix}
\uHom(\cE_0, \cF_1)(-1) \\ \oplus \\
\uHom(\cE_1, \cF_1)  
\end{matrix}
\leftarrow
\cdots
$$
determine a contracting homotopy for this kernel, proving that
\eqref{E813b} is a quasi-isomorphism.

For part (4), consider the local-to-global spectral sequence
\[ E_2^{p,q} = H^p( Y, \cExt^q_{\cO_Y}(\cM ,\cN) ) \Longrightarrow
\Ext^{p+q}_{\cO_Y}(\cM,\cN). \]
Since $\cExt^{2q+2}_{\cO_Y}(\cM ,\cN)
\cong \cExt^2_{\cO_Y}(\cM ,\cN)(q)$ for all $q \geq 0$ and similarly
for odd indices, the spectral sequence degenerates for $q \gg 0$ by
Serre's Vanishing Theorem.
\end{proof}

Recall from \ref{twisted_periodic_def} that a twisted periodic complex of sheaves is a complex of
locally free coherent sheaves $\cC$ together with a specified isomorphism $\alpha: \cC[2] \cong \cC(1)$.
\begin{lem}
\label{hom_global_sections_commute_high_enough}
Let $Y$ be any scheme that is projective over a Noetherian ring and
$\cL = \cO_Y(1)$ the corresponding very ample line bundle (we have in mind the case when $Y \into
X$ is the zero subscheme of $W$).
Let $\cC$ be a twisted
periodic complex on $Y$. For $i \gg 0$ there are isomorphisms
\[ H^i \Gamma( Y, \cC ) \cong \Gamma(Y, \cH^i(\cC) ), \]
where $\Gamma(Y,\cC)$ is $\Gamma(Y, -)$ applied degree-wise to $\cC$, and
$\cH^i$ is the $i$th cohomology sheaf of $\cC$.
\end{lem}

\begin{proof}
Let $\cB^i$ and  $\cZ^i$ be the image and 
 kernel sheaves of $\partial^i_{\cC}$, respectively. 

Since $\cZ^i \cong \cZ^{i-2}(1)$, $\cB^i \cong \cB^{i-2}(1)$ and $\cL$ is
very ample, the
higher sheaf cohomology groups of $\cZ^i, \cB^i$ vanish for $i \gg
0$ by Serre Vanishing. It follows that, for $i \gg 0$, the exact sequences
\[ 0 \to \cZ^i \to \cC^i \xra{\partial^i_\cC} \cB^i \to 0 \, \text{ and
} \,
0 \to \cB^i \to \cZ^{i+1} \to \cH^{i+1} \to 0 \]
remain exact upon applying $\Gamma(-, Y)$ and thus
\[ \Gamma(Y, \cH^{i+1}) \cong \frac{\Gamma( Y, \cZ^{i+1}
  )}{\Gamma(Y, \cB^i ) } \cong \frac{ \ker \Gamma( Y, \partial^{i+1}_\cC
  )}{\operatorname{Im} \Gamma(Y,\partial^i_\cC)} = H^{i+1} \Gamma( Y, \cC
).\qedhere\]
\end{proof}

\begin{lem} \label{lem813}
Assume $X$ is a Noetherian separated scheme with enough locally frees
(i.e., every coherent sheaf on $X$ is the quotient of a
locally free coherent sheaf),
$\cL$ is a line bundle, and $W$ is a regular global section of $\cL$.
Let $\bE$ be a matrix factorization, let $Y \into
X$ be the zero subscheme of $W$, and let $\cN$ be any
coherent sheaf on $Y$. The map induced by the canonical functor $D^b(Y)
\to \Dsing(Y)$ 
\begin{equation} \label{E813b}
\Hom_{D^b(Y)}(\coker \bE, \cN[m]) \to
\Hom_{\Dsing (Y)}(\coker \bE, \cN[m])
\end{equation}
is an isomorphism for $m \gg 0$. 
\end{lem}

\begin{proof}
We adapt the proof of
\cite[1.21]{MR2101296}. (We cannot apply this result directly
since $Y$ need not be Gorenstein.)

We first show \eqref{E813b} is onto for $m \gg 0$.
An element of
$\Hom_{\Dsing (Y)}(\coker \bE, \cN[m])$ is represented by a diagram in $D^b(Y)$ 
\begin{equation} \label{E813}
\coker(\bE) \xla{s} A \xra{f} \cN[m]
\end{equation}
such that $\cone(s)$ is perfect. The first step is to show that we can
replace $A$ by a coherent sheaf that is the cokernel of a matrix
factorization.

By Lemma \ref{sheaf_ext}(1) the
sheaf $\coker \bE$ admits a right resolution by locally free coherent
sheaves on $Y$:
\[ 0 \to \coker \bE \to Q^1 \to Q^2 \to \ldots\]
where $Q^1 = i^* \cE_1(1), Q^2 = i^* \cE_0(1), \ldots$. For any $k \geq 1$ we thus
have a distinguished triangle in $D^b(Y)$
$$
\cF_k[-k-1]  \to \coker(\bE) \to Q^{\leq k} \to \cF_k[-k]
$$
where $\cF_k = \coker(Q^k \to Q^{k+1})$. We claim there is an integer $k_0$ such that for $k \geq k_0$,
the composition $\cF_k[-k-1] \to \coker(\bE) \to \cone(s)$ is
zero. Since $\cF_k$ is the cokernel of a matrix factorization, namely
$\cF_k = \coker(\bE[k])$, 
this claim follows from the more
general claim:
given a perfect
complex $\cP^\bu = \ldots \cP^n \to \cP^{n+1} \to \ldots$ there is an integer $k_0$ such that
$\Hom_{D^b}(\coker(\bF)[-k-1], \cP^\bu) = 0$ for all 
matrix factorizations $\bF$ and all integers
$k \geq k_0$. To prove this, observe fist that, since $X$ has enough locally frees, we may
assume $\cP^\bu$ is a bounded complex of locally free coherent sheaves.
Now observe 
that $\cExt^j(\coker(\bF), \cO_Y) = 0$ for all $j > 0$;  this is a local
assertion and locally it is the statement that the transpose of a
matrix factorization is also a matrix factorization.
It follows that
$$
\cExt_Y^j(\coker \bF, \cP^n) 
\cong
\cExt_Y^j(\coker \bF, \cO_Y) \otimes_{\cO_Y} \cP^n
= 0
$$ 
for all $j > 0$ and all
$n$, and then a standard argument gives the claim.
(See the proof of
\cite[1.18]{MR2101296}.)

Consider the triangle $A \xra{s} \coker(\bE) \to \cone s \to A[1]$.  
Since for $k \geq k_0$ the composition $\cF_k[-k-1] \to \coker(\bE) \to
\cone(s)$ is zero, there exists a map $g:\cF_k[-k-1] \to A$ that
makes the left triangle commute:
$$
\xymatrix{
& \cF_k[-k-1] \ar[dr]^0 \ar[d] \ar@{-->}[dl]_{g}  &  &
\\
A \ar[r]^(.45)s & \coker(\bE) \ar[r] & \cone(s) \ar[r] & A[1] \\
}
$$
Since $\cone(\cF_k[-k-1] \to \coker(\bE)) = Q^{\leq k}$ is perfect,
composing with $g$ shows that the element of
$\Hom_{\Dsing (Y)}(\coker \bE, \cN[m])$ represented by \eqref{E813} is 
also represented by a diagram of the form
$$
\coker(\bE) \leftarrow \cF_k[-k-1] \xra{f \circ g} \cN[m]
$$ 
for $k \geq k_0$.

Now consider the triangle $Q^{\leq k}[-1] \to \cF_k[-k-1]  \to
\coker(\bE) \to Q^{\leq k} $.
For any $k$,  there is an integer $m_0 = m_0(k)$ such that 
\[\Hom_{D^b(Y)}(Q^{\leq k}[-1], \cN[m]) = 0\] for all $m \geq m_0$ --- this holds because
$\Ext^i(Q^t, \cN) \cong H^i(Y, (Q^t)^* \otimes \cN) = 0$ for all $t$
and all $i \gg 0$. It follows that the composition of $Q^{\leq k}[-1]
\to \cF_k[-k-1] \xra{f \circ g} \cN[m]$ is zero for $m \geq m_0$ and hence
the map $\cF_k[-k-1] \to \cN[m]$
factors through $\cF_k[-k-1] \to \coker(\bE)$. The element represented by
\eqref{E813} is thus actually represented by a map in $D^b(Y)$. We
have proven that
\eqref{E813b} is onto for 
$k \geq k_0$ and $m \geq m_0(k)$.

Suppose now $f: \coker(\bE) \to \cN[m]$ is a morphism in $D^b(Y)$ that
determines the zero map in $\Dsing(Y)$. Then there is a map $s: A \to
\coker(\bE)$ such that $\cone(s)$ is perfect and such that $A \to
\coker(\bE) \xra{f} \cN[m]$ is the zero map --- i.e., the image of $f$ in
$\Dsing(Y)$ is represented by a diagram of the form $\coker( \bE )
\xla{s} A \xra{0} \cN[m]$.
The
argument above shows that there is an integer $k_0$ such that 
we may take $A = \cF_k[-k-1]$ and $\cone(s) =
Q^{\leq k}$ for $k \geq k_0$. In other words, for $k \geq k_0$, the map $f$ factors as
$\coker(\bE) \to Q^{\leq k} \to  \cN[m]$ in $D^b(Y)$. But, as shown above, for $m
\geq m_0(k)$, we have $\Hom_{D^b(Y)}(Q^{\leq k},\cN[m]) = 0$. This
proves that \eqref{E813b} is one-to-one for $m \gg 0$.
\end{proof}

\begin{rem}
 Any scheme $X$ that is projective over a ring $Q$ will have
 enough locally free sheaves. Indeed, we may assume that $X =
 \bP_Q^m = \Proj Q[T_1, \ldots, T_m]$ for some ring $Q$ and $m \geq 0$.  Then any coherent sheaf
 $\cF$ is isomorphic to $\widetilde{M}$ for some finitely generated graded $Q[T_1, \ldots, T_m]$-module
 $M$. There exists a surjection $E \onto M$ with $E$ a finitely
 generated graded free
 $Q[T_1, \ldots, T_c]$-module, and the associated map $\widetilde{E} \onto
 \widetilde{M}$ gives the required surjection from a
 locally free coherent sheaf on $X$ onto $\cF$.
\end{rem}

For the Lemma below, it may help to keep in mind the case when $\cN$ is the
cokernel of a matrix factorization $\bF$. In this case, by
Lemma \ref{sheaf_ext}(3) we have that
\[H^m \Gamma( Y, \uHom_{\cO_Y}( i^* \bE, \cN ) ) \cong
H^m \Gamma(X, \uHom_{MF}(\bE,\bF)), \] and the right hand side is the set of strict morphisms between
$\bE[-m]$ and $\bF$.

\begin{lem}
\label{hom_sing_isom}
Let $X$ be a scheme that is projective over a Noetherian ring,
$\cL = \cO_X(1)$ the corresponding very ample line bundle, and $W$ a
regular global section of $\cL$. 
Let $\bE$ be a matrix factorization, let $Y \into
X$ be the zero subscheme of $W$, and let $\cN$ be any
coherent sheaf on $Y$. For every $m \in \Z$, there is a map, natural
with respect to $\bE$ and $\cN$,
\[ H^m \Gamma( Y, \uHom_{\cO_Y}( i^* \bE, \cN ) ) \to \Hom_{\Dsing(Y)}((\coker \bE)[-m],
\cN)\]
that is an isomorphism for $m \gg 0$. 
\end{lem}

\begin{proof}
Any element of $H^m
\Gamma( Y, \uHom_{\cO_Y}( i^* \bE, \cN ) )$ gives a morphism of sheaves
\[\coker(\bE[-m]) \to \cN\] which we consider as a morphism in
$\Dsing(Y)$. By \cite[3.12]{Polishchuk:2010ys}, there is a functorial isomorphism $\coker (\bE[-m] ) \cong
(\coker \bE)[-m]$ in $\Dsing(Y)$. This gives the map.

For $m \gg 0$ we have the following chain of isomorphisms:
$$
\begin{aligned}
\Hom_{D^b(Y)}(\coker \bE[-m], \cN) &  \cong
\Ext^m_{\cO_Y}(\coker \bE, \cN)  \cong \Gamma(Y, \cExt^m_{\cO_Y}(\coker \bE
,\cN)) \\
& 
\cong \Gamma(Y,\cH^m( \uHom_{\cO_Y}({i^* \bE}, {\cN}) ) ) \cong
H^m \Gamma( Y, \uHom_{\cO_Y}({i^* \bE}, {\cN})).  
\end{aligned}
$$
The second is given by \ref{sheaf_ext}(4), the
third by \ref{sheaf_ext}(1), and the fourth by \ref{hom_global_sections_commute_high_enough}.
One can check that the diagram
\[ \small{\xymatrix{ \Hom_{D^b(Y)}(\coker \bE[-m], \cN) \ar[d] \ar[r]^\cong &
  H^m \Gamma( Y, \uHom_{\cO_Y}({i^* \bE}, {\cN})) \ar@/^/[dl]\\
\Hom_{\Dsing(Y)}( (\coker \bE)[-m],\cN) & } }
\] 
commutes,
where the vertical map is induced by the canonical functor
$D^b(Y) \to \Dsing(Y)$ and the diagonal map is defined
above. The result now follows from Lemma \ref{lem813}, which shows 
that the vertical arrow is an isomorphism for $m \gg 0$.
\end{proof}

\begin{thm}
\label{hom_sing_cat}
Let $X$ be a scheme that is projective over a Noetherian ring,
$\cL = \cO_X(1)$ the corresponding very ample line bundle, and $W$ a
regular global section of $\cL$. 
Let $\bE$ be a matrix factorization, let $Y \into
X$ be the zero subscheme of $W$, and let $\cN$ be any
coherent sheaf on $Y$. For every $m \in \Z$ there is an isomorphism,
\[ \bH^m( Y, \uHom_{\cO_Y}( i^* \bE, \cN )) \xra{\cong} \Hom_{\Dsing(Y)}( \coker \bE[-m],
\cN ) \]
that is natural in both $\bE$ and $\cN$, and makes the following diagram commute:
\begin{equation}
\label{E416b}\xymatrix{ H^m\Gamma( Y, \uHom_{\cO_Y}( i^* \bE, \cN )) \ar[d]
  \ar@/^/[dr]& \\
\bH^m( Y, \uHom_{\cO_Y}( i^* \bE, \cN )) \ar[r]^(.45)\cong & \Hom_{\Dsing(Y)}( \coker \bE[-m],
\cN ).}
\end{equation}
The vertical map in the diagram is the canonical one
and the diagonal map is the map defined in Lemma \ref{hom_sing_isom}.
\end{thm}

\begin{proof}
 Consider the vertical and diagonal maps in \eqref{E416b}, both of
 which are natural in both arguments.
By Lemma \ref{hom_sing_isom} the diagonal map is an isomorphism for $m \gg 0$.
For the vertical map, 
we use the spectral sequence
$$
E_2^{p,q} = H^p(Y,\cH^q(\uHom_{\cO_Y}( i^* \bE, \cN )))
\Longrightarrow \bH^{p+q}(X, \uHom_{\cO_Y}( i^* \bE, \cN )).
$$
Since $\bE$ is a matrix factorization, we have that \[\cH^{2q+2}( \uHom_{\cO_Y}(
i^* \bE, \cN )) \cong \cH^2( \uHom_{\cO_Y}( i^* \bE, \cN ))(q)\] and
similarly in the odd case. Thus by Serre's Vanishing Theorem, the
spectral sequence above degenerates for
$q \gg 0$, giving isomorphisms
\[ \Gamma(Y, \cH^q( \uHom_{\cO_Y}( i^* \bE, \cN ) ) ) \cong \bH^q( Y,
\uHom_{\cO_Y}( i^* \bE, \cN ) ). \] 
By Lemma \ref{hom_global_sections_commute_high_enough}, for $q \gg 0$,
there is an isomorphism \[\Gamma( Y, \cH^q(
\uHom_{\cO_Y}( i^* \bE, \cN ) ) ) \cong H^q(\Gamma(Y, \uHom_{\cO_Y}( i^* \bE, \cN ) ) ) \] which shows that the left-hand map in
\eqref{E416b} is an isomorphism for $q \gg 0$.

We have proven that there is an integer $M$ such that there exists an
isomorphism
\begin{equation} \label{E418}
\bH^m( \uHom_{\cO_Y}( i^* \bE, \cN )) \to 
\Hom_{\Dsing(Y)}(\coker \bE[-m], \cN)
\end{equation}
causing \eqref{E416b} to commute,
for all $m \geq M$.
Since $X$ is a subscheme of $\Proj Q[x_1, \dots, x_n]$ for some
Noetherian ring
$Q$ and integer $n$, we have an associated Koszul exact sequence
$$
0 \to \cO_X \to \cO_X(1)^n \to \cO_X(2)^{n \choose 2} \to \cdots \to
\cO_X(n)^{n \choose n} \to 0.
$$
Define $\cP_j$ for $j = 1, \dots, n$ to be the kernel of the map
$\cO_X(j)^{n \choose  j} \to \cO_X(j+1)^{n \choose j+1}$ in this
sequence. We have an exact sequence of locally free sheaves
$$
0 \to \cP_j \to \cO(j)^{n \choose j} \to \cP_{j+1} \to 0
$$
for each $j = 1, \dots, n$, from which we obtain an exact sequence of
matrix factorizations
\[
0 \to \bE \otimes \cP_j \to \bE \otimes \cO(j)^{n \choose j}\to \bE \otimes \cP_{j+1} \to 0.
\]
We claim there exists an isomorphism
$$
\bH^0 \uHom_{\cO_Y}((i^* \bE \otimes \cP_j)[-m], \cN) \to 
\Hom_{\Dsing(Y)}(\coker \bE[-m] \otimes \cP_j, \cN)
$$
for all $m \geq M - 2j$ and for each $j = 1, \dots,
n$, making the evident analogue of \eqref{E416b} commute. 
To see this,
note that $(i^* \bE)(j) \cong (i^* \bE)[2j]$. Thus there exists an isomorphism
$$
\bH^0 \uHom_{\cO_Y}((\bE \otimes \cO(j)^{n \choose j})[-m], \cN )
\to
\Hom_{\Dsing(Y)}((\coker \bE \otimes \cO(j)^{n \choose
  j})[-m], \cN )
$$
for all $m \geq M - 2j$, making the evident analogue of \eqref{E416b} commute. The claim follows
immediately by descending induction on $j$. 

But $\cP_1 = \cO_X$ and so $i^* \bE \otimes \cP_1 = i^* \bE$, from which we
deduce that isomorphisms as in \eqref{E418}
exist for all $m \geq M - 2$, having started from the assumption
that there existed isomorphisms for $m \geq M$. Clearly such
isomorphisms exist then for all $m$.
\end{proof}

\section{Relating matrix factorizations and the singularity category}
The goal of this section is to prove Theorem \ref{introthm2} of the
introduction. We continue to assume $X$ is a Noetherian separated scheme, 
$\cL$ is a line bundle on
$X$, and $W \in \Gamma(X, \cL)$ is a global
section. Define $Y \into
X$ to be the zero subscheme of $W$.

For a matrix factorization
$$
\bE  = \left(\cE_1 \map{e_1} \cE_0 
  \map{e_0} \cE_1(1)\right),
$$ 
we view $\coker(\bE) := i^* \coker(e_1)$ as an object of
$\Dsing(Y)$. This assignment is natural in $\bE$, so gives a functor
$$
\begin{aligned}
\coker: [MF(X, \cL, W)]_\naive & \to \Dsing(Y) \\
\bE \phantom{XXX} & \mapsto \coker(\bE).
\end{aligned}
$$
By
\cite[3.12]{Polishchuk:2010ys} this is a triangulated functor.

\begin{lem} For a Noetherian separated scheme $X$, line bundle $\cL$,
  and a regular global section $W$ of $\cL$, 
if $\bE$ is a locally contractible matrix factorization, then
  $\coker(\bE) = 0$ in $D_{sing}(Y)$.
\end{lem}

\begin{proof} 
It is enough to show
that $\coker(\bE)_x$ is a free $\cO_{Y,x}$ module for all $x \in Y$.
 By
assumption $\bE_x = 0 \in HMF( \cO_{X,x}, W_x)$ for all $x \in X$, and
so the result follows from
Proposition \ref{affine_loc_acyclic}.
 \end{proof}

By the universal property of localization we immediately obtain:
\begin{prop} \label{Prop6-9}
For a Noetherian separated scheme $X$,  a line bundle
  $\cL$ and a regular global section $W$ of $\cL$, 
there is a triangulated functor
$$
\ocoker: [MF(X, \cL, W)]
  \to \Dsing(Y)
$$
such that the composition 
$$
[MF(X, \cL, W)]_\naive \to
[MF(X, \cL, W)] 
  \xra{\ocoker} \Dsing(Y)
$$
is the functor $\coker$. 
\end{prop}

\begin{thm} \label{MainThm2}  
Let $X$ be a scheme that is projective over a Noetherian ring,
$\cL = \cO_X(1)$ the corresponding very ample line bundle, and $W$ a
regular global section of $\cL$.
Define $Y \into
X$ to be the zero subscheme of $W$. Then the 
triangulated functor
$$
\ocoker: [MF(X, \cL, W)] \to \Dsing(Y)
$$
is fully faithful. 
\end{thm}

\begin{proof}
We will show there is a fully faithful functor \[\Pprime:
[MF]_\bH \to \Dsing(Y)\] such that $\Pprime \circ H = \coker$, where
$H: [MF]_\naive \to [MF]_\bH$ is the functor defined in
\eqref{E815}. To see
this implies the Theorem, let $F: [MF]_\naive \to [MF]$ be the
canonical functor, and $G: [MF] \to [MF]_\bH$ the equivalence of Theorem
\ref{MainThm1}; note that $H = G \circ F$. Consider the commutative diagram:
\[ 
\xymatrix{
[MF]_\naive \ar[r]^(.55)F \ar[dr]_{\coker} & [MF] \ar[r]_(.45)\cong^(.45)G
\ar[d]_(.45){\ocoker} & [MF]_\bH \ar[dl]^{\Pprime} \\
& \Dsing(Y) & 
}
\]
Since $\Pprime \circ
G \circ F = \coker = \ocoker \circ F$, we have
that $\Pprime \circ G \cong \ocoker$ by the universal property of Verdier
quotients. Thus if $\Pprime$ is fully faithful, so will be $\ocoker$.

To define the functor $\Pprime$, we set
$\Pprime( \bE) = \coker \bE$. By Lemma \ref{sheaf_ext}(3)
there is a quasi-isomorphism
$$
\HomMF ({\bE},
{\bF} ) \xra{\sim} i_* \uHom_{\cO_Y}( i^* \bE, \coker \bF )
$$
that is natural in both arguments, and thus there is
a natural isomorphism 
$$
\bH^0(X, \HomMF ({\bE},
{\bF} )) \cong \bH^0( Y, \uHom_{\cO_Y}( i^* \bE, \coker \bF )). 
$$
By Theorem
\ref{hom_sing_cat} there is an isomorphism
\[ \bH^0( Y, \uHom_{\cO_Y}( i^* \bE, \coker \bF)) \cong
\Hom_{\Dsing(Y)}(\coker \bE, \coker \bF)\]
that is natural in both arguments. 
By composing these, we obtain a natural map
\[
\Hom_{[MF]_\bH}(\bE, \bF) = \bH^0(X, \HomMF ({\bE},
{\bF} )) \to \Hom_{\Dsing(Y)}(\coker \bE, \coker \bF)
\] 
determining the functor 
$\Pprime$, which is a fully faithful since this map is bijective.

We have left to prove that $\Pprime \circ H = \coker$. Both sides
agree on objects, but we need to show they induce the same map on morphisms. Consider the diagram:
\[ \small{\xymatrix{  H^0 \Gamma(X, \HomMF ({\bE},
{\bF} )) \ar[d]_H \ar[r] & H^0\Gamma( Y, \uHom_{\cO_Y}( i^* \bE,
\coker \bF)) \ar[d] \ar@/^/[dr] &\\
\bH^0(X, \HomMF ({\bE},
{\bF} )) \ar[r] &\bH^0( Y, \uHom_{\cO_Y}( i^* \bE, \coker \bF)) \ar[r] & \Hom_{\Dsing(Y)}( \coker \bE,
\coker \bF )
}} \]
The right triangle is
commutative by Theorem \ref{hom_sing_cat}. The left hand square is
commutative as the horizontal maps are induced by the map of chain
complexes \[\HomMF ({\bE},
{\bF} )\to i_* \uHom_{\cO_Y}( i^* \bE, \coker \bF )\] and $H^0(-) \to
\bH^0(-)$ is a natural transformation. The composition of the
bottom arrows in the diagram is by definition the map on morphisms
induced by the functor $\Pprime$. Thus by the commutativity of the
diagram, the map on morphisms induced by $\Pprime \circ H$ is the equal
to the composition of the top two arrows of the diagram. To finish the proof
it is now enough to show that the following diagram commutes:
\[ \xymatrix{
\Hom_{[MF]}( \bE, \bF) =  H^0 \Gamma(X, \HomMF ({\bE},
{\bF} )) \ar@/^/[dr]^(.6){\coker} \ar[d] & \\
H^0\Gamma( Y, \uHom_{\cO_Y}( i^* \bE,
\coker \bF)) \ar[r] & \Hom_{\Dsing(Y)}( \coker \bE,
\coker \bF )
} \] 
This follows from the definition of the map $\HomMF ({\bE},
{\bF} )\to i_* \uHom_{\cO_Y}( i^* \bE, \coker \bF )$ given in (the
proof of) \ref{sheaf_ext}(3).
\end{proof}

\begin{rem}
Theorem~\ref{MainThm2} has been proved in
\cite[3.14]{Polishchuk:2010ys} in case $X$ is a smooth stack. Orlov
proves an analogue in \cite{1101.4051} without assuming $X$ is
projective over an affine scheme. However he assumes that $\cL =
\cO_X$. His Theorem
does not seem to imply ours, nor does ours imply his, as $\cO_X$ is
not very ample in general. Lin and Pomerleano, in \cite{1101.5847},
give a different prove of Orlov's Theorem in case $X$ is defined over
$\C$ and smooth. Finally
Positselski, in \cite{1102.0261},
proves an analogue of Theorem \ref{MainThm2} using exotic derived categories,
which does not require $X$ to be projective. He does not require the coherent sheaves in the definition of
matrix factorization to be locally free.
\end{rem}

Theorem \ref{MainThm3} will show that the objects defined below are
\emph{exactly} the cokernels of matrix factorizations in
$\Dsing(Y)$, under the assumptions of Theorem \ref{MainThm2}.
\newcommand{\rperf}[2]{\operatorname{RPerf}(#1 \into #2)}
\begin{defn}

\label{defn_rperf}
Let $i: Y \into X$ be a closed immersion of finite flat dimension.
An object $\cF$ in $D^b(Y)$ is
\emph{relatively perfect} on $Y$ if
$i_* \cF$ is perfect on $X$. We write $\rperf Y X$ for the full
subcategory of $D^b( Y)$ whose objects are relatively perfect
on $X$.

Since $i$ has finite flat dimension, $\Perf(Y)$ is a thick subcategory
of $\rperf Y
X$. We define the {\em relative singularity category of} $i$ to be
the Verdier quotient
$$
\rDsg(Y \into X) := \frac{\rperf Y X}{\Perf(Y)}.
$$
The canonical functor
$$ 
\rDsg(Y \into X) \to \Dsing(Y)
$$
is fully faithful and we may thus identify $\rDsg(Y \into X) $ with
a full subcategory of $\Dsing(Y)$.
\end{defn}

\begin{rem}
In this same context, a relative singularity category has also been defined by Positselski
in \cite{1102.0261}. In general these two categories need not
be equivalent, but we point out a relation between the two in \ref{comparison_rs}.
\end{rem}

The following establishes one half
of Theorem \ref{MainThm3} (under milder assumptions):
\begin{lem}
\label{rel_perf_mf} Assume $X$ is a Noetherian separated scheme of
finite Krull dimension and that $X$ has enough
locally frees (i.e., every coherent sheaf on $X$ is the quotient of a
locally free coherent sheaf). Let $\cL$ be a line bundle on $X$,
assume $W$ a
regular global section of $\cL$ and 
let $i: Y \into X$ be the zero subscheme of $W$. For
every object $\cG$ of 
$\rDsg(Y \into X)$, there exists a matrix factorization $\bE$ and an isomorphism
$\cG \cong \coker \bE$ in $\rDsg(Y \into X)$.
\end{lem}

\begin{proof}
Let $\cF$ be a right bounded complex of
locally free coherent sheaves on $Y$ that maps quasi-isomorphically to $\cG$. Such a
complex exists since $X$, and hence $Y$,  has enough locally frees. 
Let $\cF^{\leq k}$ denote the brutal truncation of $\cF$ in
degree $k$. For any $k$, the cone of the canonical map
$\cF \to \cF^{\leq k}$ is a perfect complex and 
hence $\cF \to \cF^{\leq k}$ is an
isomorphism in $\rDsg(Y \into X)$.
Taking $k \ll 0$, the complex $\cF^{\leq
  k}$ is exact except in degree $k$, and hence we have reduced to the
case where $\cG = \cM[-k]$ for
some coherent sheaf $\cM$ and integer $k$. Since $\coker$ is triangulated, we may
assume $k = 0$.

Since $i_* \cG$ is a perfect complex, $i_* \cM$ is locally of finite
projective dimension. In fact, for all $x \in X$, the projective
dimension of $i_* \cM_x$ as a $\cO_{X,x}$-module is at most $d :=
\dm(X)$. Consider again a resolution $\cF$ of $\cM$ by locally free coherent
sheaves on $Y$. Since  a locally free coherent sheaf
on $Y$ is locally of projective dimension one as a coherent sheaf on
$X$, a high enough 
syzygy of this resolution of $\cM$ will also be locally of projective
dimension one on $X$. Specifically, the only non-zero cohomology sheaf
of
$i_* \cF^{\leq -d}$ will be locally of projective dimension one on
$X$, and since $\cF^{\leq -d} \cong \cM = \cG$ in $\rDsg(Y \into X)$,
we may assume $\cG = \cM$ where $\cM$ is a coherent sheaf on $Y$ such that $i_*\cM$ is
locally of projective dimension one on $X$.

Now consider any surjection $\cE_0
\onto i_* \cM$ with $\cE_0$ a locally free coherent sheaf on
$X$. Since $i_* \cM$ is locally of projective dimension one, the
kernel $\cE_1$ of this surjection is locally free. That is,
we have a resolution of the form
$$
0 \lra \cE_1 \map{\alpha} \cE_0 \lra i_*\cM \lra 0
$$
with $\cE_0, \cE_1$ locally free on $X$. In the diagram
$$
\xymatrix{
\cE_1 \ar[r]^{\alpha} \ar[d]^W & \cE_0 \ar@{-->}[ld]_\beta \ar[r]
\ar[d]^W &  i_* \cM \ar[d]^W \\
\cE_1(1) \ar[r]^{\alpha(1)} & \cE_0(1) \ar[r] & i_* \cM(1) \\
}
$$
the right-most map is the zero map, and hence there exists diagonal
arrow $\beta$ as shown causing both triangles to commute.
This determines a matrix
factorization $\bE$ with $\coker(\bE) \cong \cM = \cG$.
\end{proof}

\begin{thm}
\label{MainThm3}
Let $X$ be a scheme that is projective over a Noetherian ring of
finite Krull dimension,
$\cL = \cO_X(1)$ the corresponding very ample line bundle, and $W$ a
regular global section of $\cL$. Define $Y \into
X$ to be the zero subscheme of $W$.
The functor $\ocoker$, defined in Proposition \ref{Prop6-9}, factors naturally through the subcategory
$\rDsg(Y \into X)$ of $\Dsing(Y)$ and the induced map is an
equivalence of triangulated categories:
$$ 
[MF(X, \cL, W)] \xra{\cong} \rDsg(Y \into X)
$$

In particular, if $X$
is regular, then we have an equivalence of triangulated categories
$$
\ocoker: [MF(X, \cL, W)] \xra{\cong} \Dsing(Y).
$$  
\end{thm}

\begin{proof}
For any matrix factorization $\bE = (\cE_1 \xra{e_1} \cE_0 \xra{e_0} \cE_1)$ there is a short exact sequence of
coherent sheaves on $X$
$$
0 \to \cE_1 \xra{e_1} \cE_0 \to i_* \coker(\bE) \to 0.
$$
Indeed, since $W$ is regular section, the composition $e_0 \circ e_1
= W : \cE_1
\to \cE_1(1)$ is injective and hence $e_1$ must be.
Thus $i_* \coker( \bE)$ is perfect and $\coker$ factors
through the subcategory $\rDsg(Y \into X)$. On the other hand, for any object
$\cF$ in $\rDsg(Y \into X)$, there is a matrix factorization $\bE$ such that
$\coker \bE \cong \cF$ by Lemma \ref{rel_perf_mf}.

The final assertion holds since every bounded complex of coherent
sheaves on a regular scheme is perfect.  
\end{proof}

\begin{rem}
 \label{comparison_rs}
Let $i: Y \into X$ be as above. Let $\cT$ be the thick subcategory of
$\Dsg(Y)$ generated by the objects $\bL i^* \cM$, where $\cM$ ranges
over all objects of $\Dsg(X)$. Positselski defines in \cite{1102.0261} the relative
singularity category of $i: Y \into X$, which we denote $\rpDsg( Y
\into X)$, to be the Verdier quotient $\Dsg(Y)/ \cT$. It is easy to
check that $\rDsg( Y \into X)$ is a subcategory of \[\cT^{\perp} := \{
\, \cN \in \rDsg(Y) \, | \, \Hom_{\Dsg(Y)}( \cM, \cN) = 0 \text{ for
  all objects } \cM \in \cT \, \}.\]
This implies, keeping in mind the definition of morphisms in $\Dsg(Y)
/ \cT$, that the composition of the functors \begin{equation}\label{full_faithful_rs}\rDsg(Y \into X) \to
\Dsg(Y) \to \Dsg(Y) / \cT =: \rpDsg( Y
\into X)\end{equation} is fully faithful. 

This fully faithful functor need not be an equivalence. Indeed, let $k$ be a field and let $Q = k[[x,y]]/(x^2)$
and $R = Q/(y^2) = k[[x,y]]/(x^2, y^2)$. Let $Y := \Spec R \into \Spec Q =: X$ be the natural
inclusion. If \ref{full_faithful_rs} were an equivalence, this would
imply that the smallest thick subcategory of $\Dsg(Y)$ containing
$\rDsg(Y \into X)$ and $\cT$ were all of $\Dsg(Y)$. However we claim that the residue field $k$ is not in $\rDsg(Y \into
X)$ or $\cT$. Indeed, it is easy to check that objects in either of
these categories have
periodic free resolutions, but the ranks of the free modules in a
minimal free resolution of $k$
grow linearly and so $k$ cannot have a periodic resolution.

\end{rem}

\end{document}